\documentclass{amsart}
\usepackage{amsfonts}

\newtheorem{theorem}{Theorem}
\theoremstyle{plain}

\newtheorem{corollary}[theorem]{Corollary}

\newtheorem{definition}[theorem]{Definition}

\newtheorem{lemma}[theorem]{Lemma}

\numberwithin{equation}{section}
\input{tcilatex}

\begin{document}
\title{Chebyshev Polynomials and Inequalities for Kleinian Groups}
\author{Hala Alaqad}
\address{Department of Mathematical Sciences\\
United Arab Emirate Universty}
\email{hala\_a@uaeu.ac.ae}
\author{Jianhua Gong}
\address{Department of Mathematical Sciences\\
United Arab Emirate Universty}
\email{j.gong@uaeu.ac.ae}
\author{Gaven Martin}
\address{Institute for Advanced Study\\
Massey University,}
\email{g.j.martin@massey.ac.nz}
\thanks{Research supported by UAEU UPAR Grant (31S315). This paper is partly
contained in the first author's PhD thesis.}

\begin{abstract}
The principal character of a representation of the free group of rank two
into $\mathrm{PSL}(2,\mathbb{C})$ is a triple of complex numbers that
determines an irreducible representation uniquely up to conjugacy. It is a
central problem in the geometry of discrete groups and low dimensional
topology to determine when such a triple represents a discrete group which
is not virtually abelian, that is a Kleinian group. A classical necessary
condition is J\o rgensens inequality. Here we use certainly shifted
Chebyshev polynomials and trace identities to determine new families of such
inequalities, some of which are best possible. The use of these polynomials
also shows how we can identify the principal character of some important
subgroups from that of the group itself.
\end{abstract}

\maketitle

\section{Introduction}

A connection between the geometry of discrete groups and certain shifted
Chebyshev polynomials is given in \cite[ Theorem 2.3]{GM0}. It relates the
principal character (defined below) of a representation of the free group $%
\langle a,b|\cdot \rangle $ of rank two to the principal characters of $%
\langle a^{n},b|\cdot \rangle $. Using this result, the authors extended J\o %
rgensen's inequality for discrete groups as well as other inequalities and
gave a quantitative measure of the fact that certain groups are isolated in
the topology of algebraic convergence. Asymptotically sharp estimates
concerning the way loxodromic elements can degenerate into parabolic
elements in sequences of discrete groups were also obtained. Here we extend
these results in different directions using a family of polynomial trace
identities first identified in \cite{GehMar}. The theory of these
polynomials is further developed in \cite{MM} where they arise from a two
variable polynomial Pell equation. We recall that the Chebychev polynomials
arise as the solution of the usual polynomial Pell equation.

\subsection{Notation and definitions}

Consider the Poincar\'{e} upper half-space model\textbf{\ }$\mathbb{H}^{3}$
of $3 $-dimensional hyperbolic space $\mathbb{H}^{3}=\left \{ \left(
x_{1},x_{2},x_{3}\right) \in 
\mathbb{R}
^{3}:x_{3}>0\right \} $ equipped with the hyperbolic metric 
\begin{equation*}
d^{2}s=\frac{dx_{1}^{2}+dx_{2}^{2}+dx_{3}^{2}}{x_{3}^{2}},
\end{equation*}%
the boundary of $3$-dimensional hyperbolic space is the Riemann\textbf{\ }%
sphere $\partial \mathbb{H}^{3}=\overline{\mathbb{C}},$ and we denote $%
\overline{{\mathbb{H}}^{3}}={\mathbb{H}}^{3}\cup \overline{\mathbb{C}}.$

Let $\mathrm{Isom}^{+}({\mathbb{H}}^{3})$ be the topological group of
orientation preserving hyperbolic isometries of ${\mathbb{H}}^{3},$ and let $%
\mathrm{M\ddot{o}b}^{+}(\overline{\mathbb{C}})$ be the topological group of
the normalized orientation preserving \emph{M\"{o}bius transformations}%
\begin{equation*}
\mathrm{M\ddot{o}b}^{+}(\overline{\mathbb{C}})\mathbb{=}\left \{ \frac{az+b}{%
cz+d}:a,b,c,d\in 
\mathbb{C}
\text{ and}\ ad-bc=1\right \} .
\end{equation*}%
Then these two topological groups are isomorphic as each isometry of $%
\mathrm{Isom}^{+}({\mathbb{H}}^{3})$ is an extension of the element of $%
\mathrm{M\ddot{o}b}^{+}(\overline{\mathbb{C}})$ (via the Poincar\'{e}
extension \cite{Beardon}), and also each M\"{o}bius transformation of $%
\mathrm{M\ddot{o}b}^{+}(\overline{\mathbb{C}})$ is the restriction of the
extended isometry of $\mathrm{Isom}^{+}({\mathbb{H}}^{3})$ on $\overline{{%
\mathbb{H}}^{3}}$. Next the \emph{\ }topological group \emph{projective
special linear group} $\mathrm{PSL}(2,%
\mathbb{C}
)=\mathrm{SL}(2,%
\mathbb{C}
)/\left \{ \pm Id\right \} ,$ is also isomorphic as a topological group to
each of the previous groups as each element gives the same M\"{o}bius
transformation on the Riemann sphere $\overline{\mathbb{C}}.$

A subgroup of $\mathrm{Isom}^{+}({\mathbb{H}}^{3})$ is called a \emph{%
Kleinian group} if it is non-elementary and discrete, where a subgroup of $%
\mathrm{Isom}^{+}({\mathbb{H}}^{3})$ is called \emph{elementary} if its
accumulation set in $\overline{{\mathbb{H}}^{3}}$ of an orbit of an
arbitrary point in ${\mathbb{H}}^{3}$ is finite. Equivalently a group is
elementary if it is a finite extension of an abelian group. The elementary
groups are classified, see for instance \cite{Beardon}. In \cite{GM00}
Gehring and Martin introduced the following three complex \emph{parameters}
defined by the traces of generators for a group $\left \langle
f,\,g\right
\rangle $ generated by two elements $f$ and $g$ in $\mathrm{PSL}%
(2,\mathbb{C}):$ 
\begin{equation*}
\gamma \left( f,\,g\right) =\mathrm{tr}([f,g])-2,\quad \beta \left( f\right)
=\mathrm{tr}^{2}\left( f\right) -4,\quad \beta (g)=\mathrm{tr}^{2}(g)-4
\end{equation*}%
where $[f,g]=fgf^{-1}g^{-1}$ is the multiplicative commutator. The triple 
\begin{equation*}
(\gamma (f,g),\beta (f),\beta (g))\in \mathbb{C}^{3}
\end{equation*}%
is called the \emph{principal character} of the representation of $\langle
a,b\rangle \rightarrow \mathrm{PSL}(2,\mathbb{C})$ defined by $a\mapsto f$
and $b\mapsto g$.

Furthermore, \cite[Lemma 2.2]{GM00} confirmed that if the parameter $\gamma
\left( f,\,g\right) \neq 0,$ then $\left \langle f,\,g\right \rangle $ is
determined uniquely up to conjugacy by the principal character. While $%
\gamma (f,g)=0$ implies that $f$ and $g$ share a common fixed point in $%
\overline{\mathbb{C}},$ and this is not possible in a Kleinian group. Notice
the normalizations of the trace have made the principal character of the
identity group equal to $(0,0,0)$. We record this discussion in the
following lemma.

\begin{lemma}
Every Kleinian group $\left \langle f,\,g\right \rangle $ generated by two
elements $f$ and $g$ of $\mathrm{Isom}^{+}({\mathbb{H}}^{3})$ is determined
uniquely up to conjugacy by its principal character, the triple of complex
parameters $\left( \gamma \left( f,\,g\right) ,\beta \left( f\right) ,\beta
(g)\right) .$
\end{lemma}

Thus, the complex parameters are an important tool in studying Kleinian
groups. The complex parameters for a two-generator group conveniently encode
various important geometric quantities such as translation length,\emph{\ }%
holonomy, and complex hyperbolic distance introduced in the following
definitions (also see \cite{GMM2}).

\begin{definition}
Suppose that $f$ is a non-identity element of $\mathrm{Isom}^{+}({\mathbb{H}}%
^{3}),$ $f$ is said to be

$(a)$ parabolic if it has a single fixed point in $\overline{\mathbb{C}}.$

$(b)$ elliptic if it has two fixed points in $\overline{\mathbb{C}}$ and a
fixed point in ${\mathbb{H}}^{3}.$

$(c)$ loxodromic if it has two fixed points in $\overline{\mathbb{C}}$ and
no fixed points in ${\mathbb{H}}^{3}.$
\end{definition}

In particular, if $f$ is an elliptic or a loxodromic element of $\mathrm{Isom%
}^{+}({\mathbb{H}}^{3}),$ then the hyperbolic line in ${\mathbb{H}}^{3}$
whose end points are the fixed points of $f$ on $\overline{\mathbb{C}}$ is
called the \emph{axis} of $f,$ denoted by \textrm{$axis$}$\left( f\right) .$
The axis of an elliptic element is its fixed point set, and the axis of a
loxodromic element its setwise fixed hyperbolic line.

\begin{definition}
\label{holonomy}Let $f$ and $g$\ be elliptic or loxodromic elements of $%
\mathrm{Isom}^{+}({\mathbb{H}}^{3}),$ and suppose that $p$ is a hyperbolic
line perpendicular to \textrm{$axis$}$(f).$

$(1)$ The hyperbolic distance between two hyperbolic lines $p$ and $f\left(
p\right) $\ is called the \emph{translation length }of $f,$ denoted by $\tau
_{f}.$

$(2)$ The dihedral angle between the plane containing \textrm{$axis$}$(f)$
and $p$ and the plane containing \textrm{$axis$}$(f)$ and $f\left( p\right) $
is called the \emph{holonomy} of $f,$ denoted by $\theta _{f}.$

$(3)$ The complex number $\delta +i\theta $ is called the \emph{complex
hyperbolic distance }between the axes of $f$ and $g$ if $\delta $ is the
hyperbolic distance between \textrm{$axis$}$(f)$ and \textrm{$axis$}$(g)$
and $\theta $ is the holonomy of the element of $\mathrm{M\ddot{o}b}^{+}(%
\overline{\mathbb{C}})$ whose natural extension moves \textrm{$axis$}$(f)$
to \textrm{$axis$}$(g)$, and whose axis contains the common perpendicular
between \textrm{$axis$}$(f)$ and \textrm{$axis$}$(g).$
\end{definition}

The easiest way to see the holonomy $\theta _{f}$ is to use a conjugacy to
arrange things so that \textrm{$axis$}$(f)$ lies on $x_{3}$-axis, then it is
simply the angle between the vertical projections to $\overline{\mathbb{C}}$
of $p$ and $f\left( p\right) $ at the origin. An elementary calculation (see
e.g. \cite{GM1}) gives a way to find the parameters of a two-generator group 
$\left \langle f,\,g\right \rangle $ in terms of geometric quantities $\tau
_{f},\theta _{f},$ and $\delta +i\theta $ as following: 
\begin{align}
\beta \left( f\right) & =4\sinh ^{2}\left( \frac{\tau _{f}+i\theta _{f}}{2}%
\right) ,\quad \beta (g)=4\sinh ^{2}\left( \frac{\tau _{g}+i\theta _{g}}{2}%
\right) ,  \label{2.5} \\
\gamma \left( f,\,g\right) & =\frac{\beta \left( f\right) \beta (g)}{4}\sinh
^{2}(\delta +i\theta ).  \label{2.6}
\end{align}

The identification of precise inequalities for discrete groups of M\"{o}bius
transformations started with J\o rgensen's famous inequality \cite{Jorg}
from $1976$, after earlier results of Shimizu from $1963$ \cite{Sh} and
Leutbecher from $1967$ \cite{Leut} which gave estimates in the important
special case when a generator is parabolic. J\o rgensen's inequality is the
first important universal constraint in studying the geometry of Kleinian
groups. We state it here to exhibit the relationship between the principal
characters and discreteness criteria.

\begin{lemma}
(J\o rgensen's inequality). Let $\langle f ,g\rangle$ be a Kleinian group.
Then 
\begin{equation}  \label{1.4}
|\gamma(f,g)|+|\beta(f)| \geq 1.
\end{equation}
The inequality is sharp. It is achieved for representations of the $(2,3,p)$%
-triangle groups, all $p\geq 7$, and the figure eight knot complement group $%
(\beta(f)=0, \gamma(f,g)=\frac{1+i\sqrt{3}}{2}$).
\end{lemma}

Other inequalities can be found in \  \cite{GM0, GehMar0, GehMar, GM00, GM1,
GMM2} and also \cite{Tan}. One can also find many applications of these
inequalities, the first due to J\o rgensen concerns the local compactness of
the space of principal characters under algebraic convergence, and the
\textquotedblleft thick and thin\textquotedblright \ decomposition of
hyperbolic 3-manifolds and so forth. If we write (\ref{1.4}) as 
\begin{equation*}
|\gamma (f,g)-\gamma _{0}|+|\beta (f)-\beta _{0}|\geq 1,\quad \gamma
_{0}=\beta _{0}=0,
\end{equation*}%
then we see that (\ref{1.4}) measures the isolation of the discrete
elementary groups with principal character $(0,0,z)$ from the non-elementary
Kleinian groups. Thus we are particularly interested in inequalities of the
form 
\begin{equation*}
|\gamma (f,g)-\gamma _{0}|+|\beta (f)-\beta _{0}|\geq \delta
\end{equation*}%
when $(\gamma _{0},\beta _{0},z)$ is the principal character of an
elementary discrete group, in particular the spherical triangle groups $%
A_{4} $, $S_{4}$ and $A_{5}$. An elementary compactness argument shows that $%
\delta >0$ is each case, but sharp bounds on $\delta $ in turn imply sharp
bounds on the distance between vertices of the trivalent singular graph of a
hyperbolic orbifold and this has implicates for the hyperbolic volume of
such spaces.

\section{Two-generator Kleinian Groups}

It has been long known that the class of two generator groups holds special
importance. This follows from another of J\o rgensen's results which shows
that a group is discrete if and only if all its two generator subgroups are
discrete. Similar results in fact hold in all dimensions and generally
negatively curved metrics, \cite{MActa,MHD}. This explains the following
emphasis on two-generator groups in what follows.

We begin with an initial observation.

\begin{theorem}
\label{Le: <f^n, g>} Let $\left \langle f,\,g\right \rangle $ be a Kleinian
group generated by two elements $f$ and $g$ in $\mathrm{Isom}^{+}({\mathbb{H}%
}^{3}).$ If $f^{n}$ is not the identity for some $n\in \mathbb{N} ,$ then $%
\left \langle f^{n},\,g\right \rangle $ is a Kleinian subgroup of $%
\left
\langle f,\,g\right \rangle .$
\end{theorem}

\begin{proof}
Certainly $\left \langle f^{n},\,g\right \rangle $ is discrete as a subgroup
of a discrete group. If $f$ is parabolic, loxodromic, that this group is
non-elementary and therefore Kleinian follows from the classification of the
elementary discrete groups (\cite[Section 5.1]{Beardon}) as long as $f^{n}$
is not the identity. Suppose $f$ is elliptic of order $p\geq 7$ and $f^{n}$
is not the identity. Then the group is again Kleinian from the
classification if $g$ is not elliptic order $q\leq 6$. If $g$ is elliptic of
lower order and $f$ is also elliptic, then the group is Kleinian unless $%
f^{n}$ (which is elliptic) and $g$ share a finite fixed point or a point at $%
\infty $ - so the axes of $f^{n}$ and $g$ meet and hence so also do those of 
$f$ and $g$ and this is not possible.
\end{proof}

\begin{corollary}
Let $\left \langle f,\,g\right \rangle $ be a Kleinian group generated by
two elements $f$ and $g$. Then $\left \langle f^{n},\,g^{m} \right \rangle $
is a Kleinian subgroup of $\left \langle f,\,g\right \rangle$, unless either 
$f^{n},\,g^{m}$ is the identity.
\end{corollary}

\begin{proof}
We apply Lemma \ref{Le: <f^n, g>}, then $\left \langle
g,\,f^{n}\right
\rangle =\left \langle f^{n},\,g\right \rangle $ is
Kleinian unless $f^{n}$ is the identity and hence $\left \langle
f^{n},\,g^{m}\right \rangle =\left
\langle g^{m},\,f^{n}\right \rangle $ is
Kleinian unless either $f^{n},g^{m}$ is the identity.
\end{proof}

\begin{corollary}
\label{Le: <(gf)^n, f>} Let $\left \langle f,\,g\right \rangle $ be a
Kleinian group generated by two elements $f$ and $g$. If $\left( gf\right)
^{n}$ is not the identity, then $\left \langle \left( gf\right)
^{n},\,f\right \rangle $ is a Kleinian subgroup of $\left \langle
f,\,g\right \rangle .$
\end{corollary}

\begin{proof}
We need only note that $\langle fg, f\rangle=\langle f,g \rangle$ from which
the claim clearly follows.
\end{proof}

\begin{corollary}
\label{Le: <[g,f]^n, f>}Let $\left \langle f,\,g\right \rangle $ be a
Kleinian group generated by two elements $f$ and $g$. If $f$ is not elliptic
of order $p\leq 6$ and $\left[ g,\,f\right] ^{n}$ is not the identity, then $%
\left \langle \left[ g,\,f\right] ^{n},\,f\right \rangle $ is a Kleinian
subgroup of $\left \langle f,\,g\right \rangle .$
\end{corollary}

\begin{proof}
Obviously, $\left \langle \left[ g,\,f\right] ^{n},\,f\right \rangle $ is a
subgroup of $\left \langle f,g\right \rangle $ and so discrete. Since $f$ is
loxodromic or parabolic or elliptic of order $p\geq 7,$ then 
\begin{equation*}
\left \langle f,\,gfg^{-1}\right \rangle =\left \langle
f,gfg^{-1}f^{-1}\right \rangle
\end{equation*}%
is a Kleinian group. We may then apply Theorem \ref{Le: <f^n, g>} to deduce
the result.
\end{proof}

Actually the result is true for $f$ elliptic of order $2\leq p\leq 6$ apart
from some special cases. Maskit examined the case $p=6$ in \cite{Maskit}.
The key issue is whether $\langle f,gfg^{-1}\rangle $ the group generated by
two elliptics is Kleinian. Again the classification tells us that it is
unless $f$ has order two, or $\langle f,gfg^{-1}\rangle $ is a finite
spherical triangle group. In this latter case there are only finitely many
possibilities that can occur.

\section{Trace Polynomials of two Parameters}

We will use Chebyshev polynomials of the first kind and their use in
calculating geometric quantities such as translation length and holonomy to
find formulae for calculating trace polynomials for the two complex
parameters $\gamma =\gamma(f,g)$ and $\beta =\beta(f) $ associated with
specific words in a Kleinian group $\left \langle f,\,g\right \rangle $.

\emph{Chebychev polynomials} of the first kind are defined by the recursion
formula 
\begin{equation}
T_{0}(z)=1,\text{ }T_{1}(z)=z,\text{ }T_{n+1}(z)=2zT_{n}(z)-T_{n-1}(z),\text{%
\quad for }n\in \mathbb{N}.  \label{recursion formula}
\end{equation}%
or by the explicit formula $T_{n}(z)=\frac{1}{2}\left( \left( z-\sqrt{z^{2}-1%
}\right) ^{n}+\left( z+\sqrt{z^{2}-1}\right) ^{n}\right) $. For example, the
first nine Chebychev polynomials are $T_{0}(z)=1$, and 
\begin{equation*}
\begin{array}{ll}
T_{1}(z)=z, & T_{2}(z)=2z^{2}-1, \\ 
T_{3}(z)=4z^{3}-3z, & T_{4}(z)=8z^{4}-8z^{2}+1, \\ 
T_{5}(z)=16z^{5}-20z^{3}+5z, & T_{6}(z)=32z^{6}-48z^{4}+18z^{2}-1, \\ 
T_{7}(z)=64z^{7}-112z^{5}+56z^{3}-7z, & 
T_{8}(z)=128z^{8}-256z^{6}+160z^{4}-32z^{2}+1.%
\end{array}%
\end{equation*}%
We also recall that the Chebychev polynomials $T_{n}$ have the defining
property 
\begin{equation}
T_{n}(\cosh (z))=\cosh (nz),\text{\quad for }n\in \mathbb{N}.
\label{Cheby by cosh}
\end{equation}

Now we start to find the formulae for the trace polynomials $\gamma
(f^{n},\,g)$ in terms of $\gamma $ and $\beta $ in the following theorem
which is known, but which we prove for the convenience of the reader.

\begin{theorem}
\label{Th: power of beta}Let $\left \langle f,\,g\right \rangle $ be a
Kleinian group generated by two elements $f$ and $g$ in $\mathrm{Isom}^{+}({%
\mathbb{H}}^{3})$ with two complex parameters $\gamma =\gamma \left(
f,\,g\right) $ and $\beta =\beta \left( f\right) ,$ where $f$ is elliptic or
loxodromic. Then, 
\begin{align}
\beta \left( f^{n}\right) & =2T_{n}\left( 1+\frac{\beta }{2}\right) -2,\text{%
\quad for }n\in \mathbb{N},  \label{beta of powerr by Tn} \\
\gamma (f^{n},\,g)& =\frac{\beta (f^{n})}{\beta }\gamma ,\text{\quad for }%
n\in \mathbb{N}.  \label{gama of power by beta}
\end{align}
\end{theorem}

\begin{proof}
$(1)$ Since $\cosh ^{2}\left( \frac{z}{2}\right) -\sinh ^{2}\left( \frac{z}{2%
}\right) =1$ and $\cosh (z)=\cosh ^{2}\left( \frac{z}{2}\right) +\sinh
^{2}\left( \frac{z}{2}\right) $ give $\cosh (z)=2\sinh ^{2}\left( \frac{z}{2}%
\right) +1,$ therefore,%
\begin{equation}
\cosh (nz)=1+2\sinh ^{2}\left( \frac{nz}{2}\right) ,\text{ }n\in 
\mathbb{N}
.  \label{sinh square to cosh}
\end{equation}

Since $f$ is elliptic or loxodromic, we may assume that $f$ has two fixed
points $0$ and $\infty ,$ then, up to conjugacy, $f$ can be represented by $%
f=\left( 
\begin{array}{cc}
\lambda & 0 \\ 
0 & \frac{1}{\lambda }%
\end{array}%
\right) \in \mathrm{PSL}(2,\mathbb{C})$ and hence $f^{n}=\left( 
\begin{array}{cc}
\lambda ^{n} & 0 \\ 
0 & \frac{1}{\lambda ^{n}}%
\end{array}%
\right) ,$ where $\lambda $ can be expressed as $e^{\frac{\tau }{2}}$ for a
suitable $\tau =\tau _{f}+i\theta _{f}.$ Thus, $\beta \left( f^{n}\right)
=(\lambda ^{n}-\frac{1}{\lambda ^{n}})^{2}=4\left( \frac{e^{\frac{n\tau }{2}%
}-e^{-\frac{n\tau }{2}}}{2}\right) ^{2}=4\sinh ^{2}\left( \frac{n\tau }{2}%
\right) .$ That is,%
\begin{equation}
\beta \left( f^{n}\right) =4\sinh ^{2}\left( \frac{n\tau }{2}\right) ,\text{
for }n\in 
\mathbb{N}
,\; \tau =\tau _{f}+i\theta _{f}.  \label{beta of power by sinh}
\end{equation}%
where $\tau _{f}$ and $\theta _{f}$\ are the translation length and the
holonomy of $f,$ respectively.

It follows from the identities (\ref{sinh square to cosh}) and (\ref{beta of
power by sinh}) that 
\begin{equation}
\cosh (nz)=1+\frac{\beta \left( f^{n}\right) }{2},\text{ for }n\in 
\mathbb{N}
.  \label{beta of power by cosh}
\end{equation}%
Applying for the defining property (\ref{Cheby by cosh}) and the previous
identity (\ref{beta of power by cosh}), that give $T_{n}\left( 1+\frac{\beta 
}{2}\right) =1+\frac{\beta \left( f^{n}\right) }{2}$ and hence%
\begin{equation}
\beta \left( f^{n}\right) =2T_{n}\left( 1+\frac{\beta }{2}\right) -2,\text{
for }n\in 
\mathbb{N}
.  \label{beta fn & Tn}
\end{equation}

$(2)$ Notice that one can represent $f^{n}$ and $g$ as the following: 
\begin{equation*}
f^{n}=\left( 
\begin{array}{cc}
\lambda ^{n} & 0 \\ 
0 & \frac{1}{\lambda ^{n}}%
\end{array}%
\right) \text{ and }g=\left( 
\begin{array}{cc}
a & b \\ 
c & d%
\end{array}%
\right) \in \mathrm{PSL}(2,\mathbb{C}),
\end{equation*}%
where $bc\neq 0$ because that $\left \langle f,\,g\right \rangle $ is
Kleinian. Thus, the commutator is 
\begin{equation*}
\left[ f^{n},g\right] =\left( 
\begin{array}{cc}
ad-bc\lambda ^{2n} & ab\lambda ^{2n}-ab \\ 
\frac{cd}{\lambda ^{2n}}-cd & ad-\frac{bc}{\lambda ^{2n}}%
\end{array}%
\right) .
\end{equation*}
$\allowbreak $Using $ad=1+bc,$ 
\begin{align*}
\mathrm{tr}\left[ f^{n},g\right] & =-bc(\lambda ^{2n}+\frac{1}{\lambda ^{2n}}%
)+2+2bc \\
& =-bc(\lambda ^{n}-\frac{1}{\lambda ^{n}})^{2}+2.
\end{align*}

Since $f$\ is non-parabolic, $\beta \left( f\right) =(\lambda -\frac{1}{%
\lambda })^{2}\neq 0$ and hence \ 
\begin{eqnarray*}
\gamma (f^{n},g) &=&-bc(\lambda ^{n}-\frac{1}{\lambda ^{n}})^{2} \\
&=&-\frac{(\lambda ^{n}-\frac{1}{\lambda ^{n}})^{2}}{(\lambda -\frac{1}{%
\lambda })^{2}}bc(\lambda -\frac{1}{\lambda })^{2} \\
&=&\frac{\beta (f^{n})}{\beta \left( f\right) }\; \gamma ,\text{ for }n\in 
\mathbb{N}
.
\end{eqnarray*}%
The result follows.
\end{proof}

For example, we have the following particular formulas. 
\begin{equation}
\begin{array}{l}
\gamma (f,g)=\gamma , \\ 
\gamma (f^{2},g)=\gamma \left( \beta +4\right) , \\ 
\gamma (f^{3},g)=\gamma \left( \beta +3\right) ^{2}, \\ 
\gamma (f^{4},g)=\gamma \left( \beta +4\right) \left( \beta +2\right) ^{2},
\\ 
\gamma (f^{5},g)=\gamma \left( \beta ^{2}+5\beta +5\right) ^{2}, \\ 
\gamma (f^{6},g)=\gamma \left( \beta +4\right) \left( \beta +3\right)
^{2}\left( \beta +1\right) ^{2}.%
\end{array}
\label{gamma of power}
\end{equation}

\bigskip

Next, we notice that if $\langle f,g\rangle$ is a Kleinian group and if $%
(gf)^{n+1}$ is not the identity, then 
\begin{equation}
\gamma ((gf)^{n}g,f)=\gamma ((gf)^{n}gf,f)=\gamma ((gf)^{n+1},f).
\label{gama of (gf)^ng to n+1}
\end{equation}%
This has the implication that $\gamma ((gf)^{n}g,f)\neq 0.$ Otherwise, if $%
\gamma ((gf)^{n}g,f)=0$, then $\gamma ((gf)^{n+1},f)=0$ which means that $%
(gf)^{n+1}$ and $f$ share a fixed point. Then $gf$ and $f$ share a fixed
point so that $\langle gf,f\rangle =\langle g,f\rangle $ is elementary, a
contradiction.

\medskip

We now develop a recursion formula for the commutators $\gamma _{n}=\gamma
(f^{n},g)$. Remarkably (\ref{3.9}) and (\ref{3.10}) show both $\beta (f^{n})$
and $\gamma (f^{n},g)$ satisfy the \emph{same} recursion relation.

\begin{lemma}
\label{betafn} Let $\left \langle f,\,g\right \rangle $ be a Kleinian group
generated by two elements $f$ and $g$ in $\mathrm{Isom}^{+}({\mathbb{H}}%
^{3}) $ with two complex parameters $\gamma =\gamma \left( f,\,g\right) $
and $\beta =\beta \left( f\right) ,$ where $f$ is elliptic or loxodromic.
Let $\beta _{n}=\beta (f^{n})$ and $\gamma _{n}=\gamma (f^{n},g)$. Then 
\begin{equation}
\beta _{n+1}=(2+\beta )\beta _{n}-\beta _{n-1}+2\beta _{1},  \label{3.9}
\end{equation}%
and 
\begin{equation}
\gamma _{n+1}=(2+\beta )\gamma _{n}-\gamma _{n-1}+2\gamma _{1}.  \label{3.10}
\end{equation}
\end{lemma}

\begin{proof}
Since $\beta _{n}=\beta (f^{n}),$ so $\beta _{0}=0$ and $\beta _{1}=\beta $.
Then using (\ref{beta fn & Tn}) and the recursion formula (\ref{recursion
formula}) we may calculate that 
\begin{eqnarray*}
\beta _{n+1} &=&2T_{n+1}\left( 1+\frac{\beta }{2}\right) -2=4\left( 1+\frac{%
\beta }{2}\right) T_{n}\left( 1+\frac{\beta }{2}\right) -2T_{n-1}\left( 1+%
\frac{\beta }{2}\right) -2 \\
&=&(2+\beta )2T_{n}\left( 1+\frac{\beta }{2}\right) -2T_{n-1}\left( 1+\frac{%
\beta }{2}\right) -2 \\
&=&(2+\beta )\left[ 2T_{n}\left( 1+\frac{\beta }{2}\right) -2\right] -\left[
2T_{n-1}\left( 1+\frac{\beta }{2}\right) -2\right] -4+2(2+\beta ) \\
&=&(2+\beta )\beta _{n}-\beta _{n-1}+2\beta .
\end{eqnarray*}%
Next we put $\gamma _{0}=0$ and $\gamma _{1}=\gamma (f,g)$ and use the above
result (\ref{gama of power by beta}) to calculate 
\begin{eqnarray*}
\gamma _{n+1} &=&\frac{\beta (f^{n+1})}{\beta }\gamma (f,g)=\frac{(2+\beta
)\beta _{n}-\beta _{n-1}+2\beta }{\beta }\gamma (f,g) \\
&=&(2+\beta )\frac{\beta _{n}}{\beta }\gamma (f,g)-\frac{\beta _{n-1}}{\beta 
}\gamma (f,g)+2\gamma (f,g) \\
&=&(2+\beta )\gamma _{n}-\gamma _{n-1}+2\gamma _{1}.
\end{eqnarray*}%
This completes the proof.
\end{proof}

In particular, we note the following for future use. Note the interesting
factor $\gamma (\gamma -\beta )$ in the even terms. This arises as it
identifies dihedral subgroups of a Kleinian group which are elementary and
have $\gamma =\beta $. 
\begin{align*}
\gamma _{0}& =0,\text{ }\gamma _{1}=\gamma \\
\gamma _{2}& =\gamma (\gamma -\beta ) \\
\gamma _{3}& =\gamma (\gamma -\beta -1)^{2} \\
\gamma _{4}& =\gamma (\gamma -\beta )(\gamma -\beta -2)^{2} \\
\gamma _{5}& =\gamma (1+3\beta +\beta ^{2}-3\gamma -2\beta \gamma +\gamma
^{2})^{2} \\
\gamma _{6}& =\gamma (\gamma -\beta )(\gamma -\beta -1)^{2}(\gamma -\beta
-3)^{2} \\
\gamma _{7}& =\gamma (-1-6\beta -5\beta ^{2}-\beta ^{3}+6\gamma +10\beta
\gamma +3\beta ^{2}\gamma -5\gamma ^{2}-3\beta \gamma ^{2}+\gamma ^{3})^{2}
\\
\gamma _{8}& =\gamma (\gamma -\beta )(\gamma -\beta -2)^{2}(2+4\beta +\beta
^{2}-4\gamma -2\beta \gamma +\gamma ^{2})^{2} \\
\gamma _{9}& =\gamma (\gamma -\beta -1)^{2}(-1-9\beta -6\beta ^{2}-\beta
^{3}+9\gamma +12\beta \gamma +3\beta ^{2}\gamma -6\gamma ^{2}-3\beta \gamma
^{2}+\gamma ^{3})^{2} \\
\gamma _{10}& =\gamma (\gamma -\beta )(5+5\beta +\beta ^{2}-5\gamma -2\beta
\gamma +\gamma ^{2})^{2}(1+3\beta +\beta ^{2}-3\gamma -2\beta \gamma +\gamma
^{2})^{2}.
\end{align*}

As a consequence of $\beta_n$ and $\gamma_n$ satisfying the same recursion
relation we also have the following easy consequence.

\begin{lemma}
Let $\beta _{n}=\beta (f^{n})$ and $\gamma _{n}=\gamma (f^{n},g)$. Then $%
\alpha _{n}=\gamma (f^{n},g)-\beta (f^{n})$ satisfies the recursion relation 
\begin{equation}
\alpha _{n+1}=(2+\beta )\alpha _{n}-\alpha _{n-1}+2\alpha _{1}.  \label{3.11}
\end{equation}
\end{lemma}

The recursion relation (\ref{3.11}) has the explicit solution 
\begin{equation*}
\alpha _{n}=-\frac{2^{-n}\left( 2^{1+n}-\left( 2+\beta -\sqrt{\beta }\sqrt{%
4+\beta }\right) ^{n}-\left( 2+\beta +\sqrt{\beta }\sqrt{4+\beta }\right)
^{n}\right) \gamma }{\beta }
\end{equation*}%
while (\ref{3.10}) has the solution 
\begin{equation*}
\gamma _{n}=-\frac{2^{-n}\alpha \left( 2^{1+n}-\left( 2+\beta -\sqrt{\beta }%
\sqrt{4+\beta }\right) ^{n}-\left( 2+\beta +\sqrt{\beta }\sqrt{4+\beta }%
\right) ^{n}\right) }{\beta }.
\end{equation*}%
We deduce that $\lambda _{n}=\gamma _{n}\alpha _{n}=\gamma _{n}(\gamma
_{n}-\beta _{n})$ is given by the formula 
\begin{equation}
\lambda _{n}=\frac{\lambda _{1}}{4^{n}\beta ^{2}}\left[ -2^{1+n}+\left(
2+\beta -\sqrt{\beta }\sqrt{4+\beta }\right) ^{n}+\left( 2+\beta +\sqrt{%
\beta }\sqrt{4+\beta }\right) ^{n}\right] ^{2}.  \label{lambda}
\end{equation}%
These calculation enable us to calculate various principal characters for
subgroups.

\begin{theorem}
Let $\langle f,g\rangle $ be a group with principal character%
\begin{equation*}
(\gamma ,\beta ,\widetilde{\beta })=(\gamma (f,g),\beta (f),\beta (g)).
\end{equation*}%
Then the subgroups below have the associated principal characters 
\begin{eqnarray}
\langle f^{n},g\rangle &\mapsto &(\gamma _{n},\beta _{n},\widetilde{\beta }),
\\
\langle f^{n},gf^{n}g^{-1}\rangle &\mapsto &(\lambda _{n},\beta _{n},\beta
_{n}).
\end{eqnarray}
\end{theorem}

We remark here that a Kleinian group generated by two elements of equal
trace admits an absolute lower bound on the commutator parameter. Thus if $%
\langle f,g\rangle $ is Kleinian we must have 
\begin{equation*}
|\gamma _{n}|\geq 0.198\cdots .
\end{equation*}

\medskip

There is an important special case where $g$ has order two which motivates
the above calculation. This is because of the following lemma \cite{GehMar}.

\begin{lemma}
\label{lem11} Let $\langle f,g\rangle $ be a Kleinian group with principal
character%
\begin{equation*}
(\gamma ,\beta ,\widetilde{\beta })=(\gamma (f,g),\beta (f),\beta (g)).
\end{equation*}%
Then there is a discrete group $\langle f,\phi \rangle $ with principal
character $(\gamma ,\beta ,-4)$, so that in particular $\phi $ has order
two. This group is also non-elementary if $f$ does not have order $p\in
\{2,3,4,5,6\}$. If $f$ does have finite order $p\leq 6$, then the following
must also be true.

\begin{itemize}
\item $f$ has order $2$ and $\langle f,\phi \rangle $ is the Klein $4$-group.

\item $f$ has order $3$ and $\langle f,\phi \rangle$ is one of the groups $%
A_4 $ or $S_4$.

\item $f$ has order $4$ and $\langle f,\phi \rangle$ is the group $S_4$.

\item $f$ has order $5$ and $\langle f,\phi \rangle$ is the group $A_5$.

\item $f$ has order $6$ and $\langle f,\phi \rangle$ is the $(2,3,6)$
Euclidean triangle group.
\end{itemize}

Further, there is also a group with principal character $(\beta -\gamma
,\beta ,-4)$.
\end{lemma}

In fact the precise values of $\gamma $ and $\beta $ in each of these cases
can be found in \cite{Martin3}. The two groups with principal characters $%
(\gamma ,\beta ,-4)$ and $(\beta -\gamma ,\beta ,-4)$ are the two $\mathbb{Z}%
_{2}$ extensions of the group $\langle f,gfg^{-1}\rangle $ generated by
elements of the same order and, in the non-parabolic case, the elements of
order two are suitable involutions in the bisector of the common
perpendicular to the lines \textrm{$axis$}$(f)$ and $g(\mathrm{axis}(f))=%
\mathrm{axis}(gfg^{-1})$.

\medskip

The virtue of Lemma \ref{lem11} is that any inequality which holds for the
triple of parameters $(\gamma ,\beta ,\widetilde{\beta })$ of a Kleinian
group must hold (with a small finite list of exceptions in the spherical and
Euclidean triangle groups) for the parameters $(\gamma ,\beta ,-4)$ . We now
exploit this.

\medskip

We recall the Fricke identity (see \cite{Fricke-Klein}) 
\begin{equation}
\mathrm{tr}[f,g]=\mathrm{tr}^{2}(f)+\mathrm{tr}^{2}(g)+\mathrm{tr}^{2}(fg)-%
\mathrm{tr}(f)\mathrm{tr}(g)\mathrm{tr}(fg)-2,  \label{Fricke}
\end{equation}%
which gives the following version of the Fricke identity in terms of our
complex parameters 
\begin{equation}
\gamma \left( f,\,g\right) =\beta \left( f\right) +\beta (g)+\beta (fg)-%
\mathrm{tr}(f)\mathrm{tr}(g)\mathrm{tr}(fg)+8.  \label{Friche beta Id}
\end{equation}%
Now if $g$ is elliptic of order $2$, then $\mathrm{tr}(g)=0$ and $\beta
(g)=-4,$ so 
\begin{equation}
\beta (fg)=\gamma -\beta -4.  \label{betafg}
\end{equation}%
Notice that since $\langle f,fg\rangle =\langle f,g\rangle $ in general, we
have the following lemma, again using Lemma \ref{lem11}.

\begin{lemma}
The triple of parameters $(\gamma ,\beta ,-4)$ is the principal character of
a representation of a two-generator Kleinian group if and only if $(\gamma
,\gamma -\beta -4,-4)$ and $(\beta -\gamma ,-\gamma -4,-4)$ are also.
\end{lemma}

By way of example for what will follow, this implies the following partly
known generalization of J\o rgensen's inequality.

\begin{lemma}
\label{lem15} Let $\langle f,g\rangle $ be a Kleinian group with principal
character $(\gamma ,\beta ,\widetilde{\beta })$ and $f$ not of order $2$.
Then both 
\begin{equation*}
|\gamma -\beta -4|+|\gamma |\geq 1\; \; \; \; \text{\textrm{and}}\; \; \;
\;|\gamma +4|+|\beta -\gamma |\geq 1.
\end{equation*}%
Both inequalities are sharp and realized in the $(2,3,7)$ hyperbolic
triangle group.
\end{lemma}

\begin{proof}
Both inequalities follow from J\o rgensen's inequality and (\ref{betafg}) if 
$\langle f,gfg^{-1}\rangle $ is Kleinian using Lemma \ref{lem11}. So we
examine the possibilities where this group is one of $A_{4},S_{4},A_{5}$, or
the Euclidean triangle group. If $f$ is order $3,4$, or $6$, then $\beta
(f)\in \{-1,-2,-3\}$ and so 
\begin{equation*}
|\gamma -\beta -4|+|\gamma |\geq |\beta +4|\geq 1.
\end{equation*}%
The same is true if $f$ is primitive elliptic of order $5$, for then $\beta
(f)+4=\frac{1}{2}\left( 3+\sqrt{5}\right) =2.61\cdots .$ If $f$ is not a
primitive elliptic of order $5$, then by Lemma \ref{lem11} $\langle
f,gfg^{-1}\rangle $ must be the group $A_{5}$. Then $gfg^{-1}f^{-1}=[f,g]$
must be elliptic of order $2,3$ or $5$. If $[f,g]$ has order $2$ or $3$,
then $\gamma \in {-1,-2,-3}$ and the inequality is trivial. We are left with
the case $\beta (f)=-4\sin ^{2}\frac{2\pi }{5}=\frac{1}{2}\left( -5-\sqrt{5}%
\right) =-3.618\cdots $, and 
\begin{equation*}
\gamma (f,g)\in \{2\cos \frac{\pi }{5}-2,2\cos \frac{2\pi }{5}%
-2\}=\{-0.381966,-1.38197\}.
\end{equation*}%
Thus $\beta (f)=\frac{1}{2}\left( -5-\sqrt{5}\right) $ and $\gamma (f,g)=%
\frac{1}{2}(-3+\sqrt{5})$. Then 
\begin{equation*}
\gamma -\beta =\frac{1}{2}(-3+\sqrt{5})+\frac{1}{2}\left( 5+\sqrt{5}\right)
=1+\sqrt{5}
\end{equation*}%
and 
\begin{equation*}
|\gamma -\beta -4|+|\gamma |=3-\sqrt{5}+\frac{1}{2}(3-\sqrt{5})=\frac{3}{2}%
(3-\sqrt{5})=1.1459\cdots \geq 1.
\end{equation*}%
A similar analysis deals with the second inequality.

The previous argument suggests where we might find equality. If $\beta =-3$,
then the first inequality, with equality, reads as $|\gamma |+|\gamma -1|=1$
and so $\gamma \in \lbrack 0,1]$. Such an example can be found in \cite[%
Theorem 4.17]{GGM} in the $(2,3,7)$ hyperbolic triangle group where%
\begin{equation*}
\gamma =4\left( \cos ^{2}\frac{2\pi }{7}-\sin ^{2}\frac{\pi }{7}\right)
=0.80193\cdots \in \lbrack 0,1].
\end{equation*}

If $\beta =-3$, then the equality corresponding the second inequality is $%
|\gamma +4|+|\gamma +3|=1$ and hence $\gamma \in \lbrack -4,-3]$. Now we
take $\gamma =-4$ from the second case in \cite[Theorem 4.17]{GGM} in the $%
(2,3,7)$ hyperbolic triangle group and hence the triangle group with
principal character $\left( -4,-3,-4\right) $ achieves the sharpness of the
second inequality.
\end{proof}

\section{Further inequalities}

We define a sequence $a_n^{u,v}$ recursively by the relation 
\begin{equation*}
a_{n+1}^{u,v} = (2+u) a_{n}^{u,v}-a_{n-1}^{u,v}+2v, \hskip10pt
a_{0}^{u,v}=0,\; \; a_{1}^{u,v}=v.
\end{equation*}
Then Lemma \ref{lem11} together with (\ref{3.9}) and (\ref{3.10}) give use
the following easy consequence.

\begin{theorem}
Let $(\gamma ,\beta ,-4)$ be the principal character of a representation of
a two-generator Kleinian group for some $\gamma ,\beta \in \mathbb{C}.$ Then
for every $n\geq 1$ 
\begin{equation}
|a_{n}^{2+\beta ,\gamma }|+|a_{n}^{2+\beta ,\beta }|\geq 1
\end{equation}%
unless $a_{n}^{2+\beta ,\beta }=0$ and $\beta =-4\sin ^{2}(\frac{p\pi }{n})$%
, for some integer $p$.
\end{theorem}

\begin{proof}
Because of (\ref{3.9}) and (\ref{3.10}) this is simply J\o rgensen's
inequality applied to the group $\langle f^{n},g\rangle $ which is Kleinian
if $f^{n}$ is not the identity (Theorem \ref{Le: <f^n, g>}). But $f^{n}$
equal to the identity implies $a_{n}^{2+\beta ,\beta }=0$ and that $f$ is
elliptic of order $n$.
\end{proof}

Then from Lemma\ref{lem15} we have the following corollary.

\begin{corollary}
Let $(\gamma ,\beta ,-4)$ be the principal character of a representation of
a two-generator Kleinian group for some $\gamma ,\beta \in \mathbb{C}.$ Then
for every $n\geq 1$ 
\begin{equation}
|a_{n}^{2+\beta ,\gamma }-a_{n}^{2+\beta ,\beta }-4|+|a_{n}^{2+\beta ,\gamma
}|\geq 1\; \mathrm{and}\;|a_{n}^{2+\beta ,\gamma }+4|+|a_{n}^{2+\beta ,\gamma
}-a_{n}^{2+\beta ,\beta }|\geq 1
\end{equation}%
unless $a_{n}^{2+\beta ,\beta }=0$ and $\beta =-4\sin ^{2}(\frac{p\pi }{n})$%
, for some integer $p$.
\end{corollary}

Now the recursion identity for the Chebychev polynomials also gives 
\begin{equation}
\beta \big((fg)^{n+1}\big)=\beta \big((gf)^{n+1}\big)=2T_{n}\left( 1+\frac{%
\beta (fg)}{2}\right) -2
\end{equation}%
Thus we also have, with 
\begin{equation*}
\tilde{\gamma}_{n}=\gamma ((fg)^{n},g),\hskip15pt\tilde{\beta _{n}}=\beta
((fg)^{n})
\end{equation*}%
\begin{eqnarray*}
\tilde{\gamma}_{n+1} &=&(\gamma _{1}-\beta -2)\tilde{\gamma}_{n}-\tilde{%
\gamma}_{n-1}+2\tilde{\gamma}_{1},\quad \text{for }n\in \mathbb{N} \\
\widetilde{\beta }_{n+1} &=&(\gamma _{1}-\beta -2)\widetilde{\beta }_{n}-%
\widetilde{\beta }_{n-1}+2\tilde{\beta}_{1},\text{\quad for }n\in \mathbb{N}
\end{eqnarray*}%
and this calculation applied to the group $\left \langle
(fg)^{n},g\right
\rangle $ gives us the following corollary.

\begin{corollary}
Let $\gamma ,\beta \in \mathbb{C}$. Then there is a Kleinian group with
principal character $(\gamma ,\beta ,-4)$ if and only if for every $n\geq 1$ 
\begin{equation}
|a_{n}^{\gamma -\beta ,\gamma }|+|a_{n}^{\gamma -\beta ,\gamma -\beta
-2}|\geq 1.
\end{equation}
\end{corollary}

\section{Trace polynomials linear in $\protect \beta$}

In \cite[Lemma 2.1]{GehMar}, the following trace polynomial of two complex
variables $\gamma $ and $\beta $ that is linear in $\beta $ is identified. 
\begin{equation}
\gamma \left( f,gfg^{-1}\right) =\gamma \left( \gamma -\beta \right) .
\label{poly}
\end{equation}%
One may observe that $\gamma \left( f,\left[ g,\,f\right] \right) =\gamma
\left( f,gfg^{-1}\right) $ and also that 
\begin{equation*}
\beta (\left[ g,\,f\right] )=\mathrm{tr}^{2}\left( \left[ g,\,f\right]
\right) -4=(\gamma \left( g,\,f\right) +2)^{2}-4=\gamma ^{2}+4\gamma .
\end{equation*}%
This reads as 
\begin{equation}
\beta (\left[ g,\,f\right] )=\gamma (\gamma +4).  \label{beta of commu}
\end{equation}

\ The following theorem shows that the trace polynomials associated the word 
$\left[ g,\,f\right] ^{n+1}$ are linear in $\beta $ for each $n\in \mathbb{N}
.$ These calculations quickly lead to the following theorem showing there
are infinitely many trace polynomials of two complex variables $\gamma $ and 
$\beta $ which are linear in $\beta .$

\begin{theorem}
\label{Th:gamma of power commu} Suppose that $\left \langle
f,\,g\right
\rangle $ is a Kleinian group generated by two elements $f$ and 
$g$ with two complex parameters $\gamma =\gamma \left( f,\,g\right) $ and $%
\beta =\beta \left( f\right) ,$ and suppose that $\left[ g,\,f\right] $ is
elliptic or loxodromic. Then the following recursion formulas for the trace
polynomials $\gamma _{n+1}=\gamma (f,\left[ g,\,f\right] ^{n+1})$ holds. 
\begin{equation*}
\gamma _{0}=0,\text{ \ }\gamma _{1}=\gamma (\gamma -\beta ),\text{ \ }\gamma
_{n+1}=(\gamma ^{2}+4\gamma +2)\gamma _{n}-\gamma _{n-1}+2\gamma (\gamma
-\beta ),\text{ for }n\in \mathbb{N} .
\end{equation*}

In particular,%
\begin{align}
\gamma _{2}& =\gamma (\gamma -\beta )(\gamma +2)^{2},  \label{gamma[g,f]^2}
\\
\gamma _{3}& =\gamma (\gamma -\beta )(\gamma +1)^{2}(\gamma +3)^{2},
\label{gamma star of power 3} \\
\gamma _{4}& =\gamma \left( \gamma -\beta \right) \left( \gamma +2\right)
^{2}\left( \gamma ^{2}+4\gamma +2\right) ^{2}, \\
\gamma _{5}& =\gamma \left( \gamma -\beta \right) \left( \gamma ^{2}+3\gamma
+1\right) ^{2}\left( \gamma ^{2}+5\gamma +5\right) ^{2}.
\end{align}
\end{theorem}

\section{Inequalities for Kleinian\emph{\ }Groups}

In this section, we apply the ideas above to establish new inequalities for
two-generator Kleinian groups $\left \langle f,\,g\right \rangle $ with two
complex parameters $\gamma =\gamma \left( f,\,g\right) $ and $\beta =\beta
\left( f\right) :$ 
\begin{equation*}
\left \vert \gamma \left( f,\,g\right) \right \vert +\left \vert \beta
\left( f\right) -\beta _{0}\right \vert \geq r,
\end{equation*}
where $\beta _{0}$ is typically real.

We first recall that the space of principal characters of Kleinian groups is
locally compact. This is basically a consequence of J\o rgensen's algebraic
convergence theorem \cite{Jorg} after some normalizations see \cite[Theorem
6.17]{Martin3}.

\begin{lemma}
Let $M\geq 0$. The set of triples $(\gamma ,\beta ,\widetilde{\beta })$ with 
$|\beta |\leq M$ and $\beta \neq -4$, and associated to a Kleinian group via 
\begin{equation*}
\gamma =\gamma (f,g),\; \; \beta =\beta (f),\; \; \widetilde{\beta }=\beta
(g)
\end{equation*}%
is compact.
\end{lemma}

In fact as $\beta \to \infty$ or $\beta=-4$ it is possible that $\gamma \to0$%
. However this is tightly controlled. We use our earlier results to quantify
this in the following theorem.

\begin{theorem}
Let $\left \langle f,\,g\right \rangle $ be a Kleinian group and $f$ not
order two. Then 
\begin{equation}
\left \vert \gamma (f^{2},g)\right \vert \geq 2-2\cos (\frac{\pi }{7}%
)=0.198\cdots .
\end{equation}%
This estimate is sharp and achieved in the $(2,3,7)$ hyperbolic triangle
group.
\end{theorem}

\begin{proof}
We recall \cite[Theorem 4.1]{Cao} which states that if $\langle u,v\rangle $
is a discrete group generated by two elements of equal trace, then 
\begin{equation}
|\gamma (u,v)|\geq 2-2\cos (\frac{\pi }{7})=0.198\cdots
\end{equation}%
unless $\gamma (u,v)\in \{0,\beta \}$ or $\beta (u)=-4$, and that this
inequality is sharp. Under our hypotheses we have $\langle fg,g\rangle $
discrete and so $\langle fg,g(fg)g^{-1}\rangle =\langle fg,gf\rangle $ is
discrete and generate by elements of equal trace. Let $\gamma =\gamma (f,g)$
and $\beta =\beta (f),$ then $\gamma (fg,g)=\gamma .$ Using (\ref{poly}),
Lemma \ref{lem11}, (\ref{betafg}), and (\ref{gamma of power}), we calculate
the parameters to be 
\begin{eqnarray*}
\gamma (fg,gf) &=&\gamma (fg,g)(\gamma (fg,g)-\beta (fg) \\
&=&\gamma (\gamma -(\gamma -\beta -4)) \\
&=&\gamma (\beta +4)=\gamma (f^{2},g).
\end{eqnarray*}%
That $\langle fg,g\rangle $ is Kleinian implies both $\gamma \neq 0$ and $%
\gamma \neq \beta $. The result now follows.
\end{proof}

\begin{lemma}
\label{Lem Chebychev ineq} Suppose that $\left \langle f,\,g\right \rangle $
is a Kleinian group generated by two elements $f$ and $g$, with $g$ of order 
$2$, principal character $\left( \gamma ,\beta ,-4\right) ,$ and $fg$ is not
parabolic. Suppose that $\left \langle f,\,g\right \rangle $ achieves the
minimum value of 
\begin{equation}
|\gamma |+|\beta -\beta _{0}|
\end{equation}%
Then 
\begin{equation*}
\left \vert \gamma -\beta -4\right \vert \leq 2\left \vert T_{n+1}\left( 
\frac{1}{2}(\gamma -\beta -2)\right) -1\right \vert ,\text{\quad for }n\in 
\mathbb{N}.
\end{equation*}
\end{lemma}

\begin{proof}
Since $\left \langle f,\,g\right \rangle $ is a Kleinian group, by Corollary %
\ref{Le: <(gf)^n, f>} and (\ref{betafg}), $\left \langle \left( gf\right)
^{n+1},f\right \rangle $ is a Kleinian group and $\beta (fg)=\gamma -\beta
-4 $. By minimality we must have 
\begin{equation*}
\left \vert \gamma \right \vert +\left \vert \beta -\beta _{0}\right \vert
\leq \left \vert \gamma \left( \left( gf\right) ^{n+1},f\right) \right \vert
+|\beta -\beta _{0}|,
\end{equation*}%
which gives $|\gamma |\leq |\gamma \left( \left( gf\right) ^{n+1},f\right)
|. $ Applying Theorem \ref{betafn}, 
\begin{equation*}
|\gamma |\leq \frac{|\beta ((fg)^{n})|}{|\beta (fg)|}|\gamma |.
\end{equation*}%
Since $\gamma \neq 0,$ dividing by $|\gamma |$ gives 
\begin{equation*}
|\beta (fg)|\leq |\beta ((fg)^{n})|.
\end{equation*}%
Assuming that $fg$ is not parabolic, so $\beta (fg)\neq 0$. Then Theorem \ref%
{Th: power of beta} shows us that this is equivalent to the inequality,%
\begin{equation*}
\left \vert \gamma -\beta -4\right \vert \leq 2\left \vert T_{n}\Big(\frac{1%
}{2}(\gamma -\beta -2)\Big)-1\right \vert .
\end{equation*}%
This is what we wanted to prove.
\end{proof}

We make a brief remark about the parabolic case.

\begin{theorem}
Let $\langle f,g\rangle$ be Kleinian and $fg$ parabolic. Then 
\begin{equation*}
|\gamma|+|\beta-\beta_0| \geq 1
\end{equation*}
\end{theorem}

\begin{proof}
The Shimitzu-Leutbecher inequality (or J\o rgensen's inequality in the
parabolic case) imply that $|\gamma(fg,g)|=|\gamma(f,g)|\geq 1$.
\end{proof}

In what follows we will typically ignore the parabolic case and leave that
to the reader.

First, we apply Lemma \ref{Lem Chebychev ineq} for the Chebychev polynomial $%
T_{2}=2z^{2}-1:$ 
\begin{align*}
\left \vert \gamma -\beta -4\right \vert & \leq 2\left \vert T_{2}(\frac{1}{2%
}(\gamma -\beta -2))-1\right \vert \\
& =\left \vert \left( \gamma -\beta -2\right) ^{2}-4\right \vert \\
& \leq \left \vert \gamma -\beta -4\right \vert |\gamma -\beta |.
\end{align*}
Thus at a minimum of $|\gamma|+|\beta|$ we have $|\gamma-\beta|\geq 1$. This
provides a new approach of J\o rgensen's inequality since 
\begin{equation*}
|\gamma |+|\beta|\geq |\gamma -\beta |\geq 1.
\end{equation*}

Next,

\begin{theorem}
\label{Th:beta+1} Suppose that $\left \langle f,\,g\right \rangle $ is a
Kleinian group generated by two elements $f$ and $g$ with principal
character $\left( \gamma ,\beta ,-4\right) $. Then 
\begin{equation}
|\gamma |+|\beta +1|^{2}\geq 1.
\end{equation}%
Hence 
\begin{equation}
|\gamma |+|\beta +1|\geq 1.
\end{equation}%
Both these inequalities are sharp. They hold with equality in the
generalized triangle group $\Gamma (6,2;3)$ with parameters $(-1,-1,-4)$.
\end{theorem}

\begin{proof}
Again the result follows from Shimitzu-Leutbecher if $fg$ is parabolic and
so we assume otherwise. Consider the Chebychev polynomial $%
T_{3}(z)=4z^{3}-3z.$ Lemma \ref{Lem Chebychev ineq} gives 
\begin{align*}
\left \vert \gamma -\beta -4\right \vert & \leq \left \vert T_{3}(\frac{1}{2}%
(\gamma -\beta -2))-1\right \vert \\
& =\left \vert (\gamma -\beta -4)\left( \gamma -\beta -1\right) ^{2}\right
\vert \\
& \leq \left \vert \gamma -\beta -4\right \vert |\gamma -\left( \beta
+1\right) |^{2}.
\end{align*}

Dividing by $\left \vert \gamma -\beta -4\right \vert ,$ we find that at the
minimum, 
\begin{equation*}
1\leq |\gamma -\left( \beta +1\right) |^{2}.
\end{equation*}%
This establishes the inequalities. The example of the generalized triangle
group can be found in Hagelberg, Maclachlan and Rosenberger \cite{HMR}.
\end{proof}

\begin{theorem}
\label{Th (0,2) ineq}Suppose that $\left \langle f,\,g\right \rangle $ is a
Kleinian group generated by two elements $f$ and $g$ with principal
character $\left( \gamma ,\beta ,-4\right)$ . Then%
\begin{equation*}
|\gamma |+|\beta +2|\geq \frac{\sqrt{5}-1}{2}.
\end{equation*}%
This inequality is sharp for the $(2,4,5)$ hyperbolic triangle group with
the parameters $\left( \frac{\sqrt{5}-1}{2},-2,-4\right) .$
\end{theorem}

\begin{proof}
Apply Lemma \ref{Lem Chebychev ineq} with the Chebychev polynomial $%
T_{4}=8z^{4}-8z^{2}+1,$ 
\begin{align*}
\left \vert \gamma -\beta -4\right \vert & \leq 2|T_{4}(\frac{1}{2}(\gamma
-\beta -2))-1| \\
& =\left \vert (\gamma -\beta -2)^{2}(\gamma -\beta )\left( \gamma -\beta
-4\right) \right \vert \\
& \leq |\gamma -\beta -2|^{2}|\gamma -\beta |\left \vert \gamma -\beta
-4\right \vert .
\end{align*}%
Dividing by $\left \vert \gamma -\beta -4\right \vert ,$%
\begin{align*}
1& \leq |\gamma -\beta -2|^{2}|\gamma -\beta | \\
& \leq |\gamma -\beta -2|^{2}\left( |\gamma -\beta -2|+2\right) .
\end{align*}

Let $x=|\gamma -\beta -2|,$ then $x\leq |\gamma |+|\beta +2|$ and hence 
\begin{equation*}
1\leq x^{2}\left( x+2\right) .
\end{equation*}
Solving the inequality gives $x\geq \frac{\sqrt{5}-1}{2},$ i.e., $\frac{%
\sqrt{5}-1}{2}\leq x\leq |\gamma |+|\beta +2|.$ So we conclude that 
\begin{equation*}
|\gamma |+|\beta +2|\geq \frac{\sqrt{5}-1}{2}.
\end{equation*}
The $(2,4,5)$ hyperbolic triangle group has the triple of parameters $\left( 
\frac{\sqrt{5}-1}{2},-2,-4\right)$ which verifies the sharpness.
\end{proof}

For the next theorem we recall $2\cos \frac{\pi }{5}-2=\frac{\sqrt{5}-3}{2}.$

\begin{theorem}
Suppose that $\left \langle f,\,g\right \rangle $ is a Kleinian group
generated by two elements $f$ and $g$ with principal character $\left(
\gamma ,\beta ,-4\right) $. Then%
\begin{equation*}
|\gamma |+\left \vert \beta +\frac{3\pm \sqrt{5}}{2}\right \vert \geq \frac{%
3-\sqrt{5}}{2}.
\end{equation*}%
This inequality is sharp for the ${\mathbb{Z}}_{2}$-extension of $(10,10,5)$
hyperbolic triangle group with the triple of parameters $\left( \frac{\sqrt{5%
}-3}{2},-\frac{3\pm \sqrt{5}}{2},-4\right) .$
\end{theorem}

\begin{proof}
We consider the Chebychev polynomial $T_{5}=16z^{5}-20z^{3}+5z$ and use the
inequality in Lemma \ref{Lem Chebychev ineq} to obtain 
\begin{align*}
\left \vert \gamma -\beta -4\right \vert & \leq 2|T_{5}(\frac{1}{2}(\gamma
-\beta -2))-1| \\
& =\left \vert (\gamma -\beta -2)^{5}-5(\gamma -\beta -2)^{3}+5(\gamma
-\beta -2)-2\right \vert .
\end{align*}%
Let $x=\gamma -\beta -2,$ then $\gamma -\beta -4=x-2\neq 0$ and hence 
\begin{align*}
\left \vert x-2\right \vert & \leq \left \vert x^{5}-5x^{3}+5x-2\right \vert
\\
& =\left \vert \left( x-2\right) \left( x^{2}+x-1\right) ^{2}\right \vert \\
& \leq \left \vert x-2\right \vert \left \vert x^{2}+x-1\right \vert ^{2}
\end{align*}%
Dividing $\left \vert x-2\right \vert $ gives $1\leq \left \vert
x^{2}+x-1\right \vert ^{2}.$ Thus,%
\begin{align*}
1& \leq |x^{2}+x-1| \\
& =\left \vert \left( x-\frac{-1+\sqrt{5}}{2}\right) \left( x-\frac{-1-\sqrt{%
5}}{2}\right) \right \vert \\
& =\left \vert \gamma -\beta -\frac{3+\sqrt{5}}{2}\right \vert \left \vert
\gamma -\beta -\frac{3-\sqrt{5}}{2}\right \vert ,
\end{align*}%
which gives the inequality%
\begin{equation}
\left \vert \gamma -\beta -\frac{3+\sqrt{5}}{2}\right \vert \left \vert
\gamma -\beta -\frac{3-\sqrt{5}}{2}\right \vert \geq 1.  \label{pre-inequ}
\end{equation}%
Rearranging the first factor of the previous inequality,%
\begin{align*}
1& \leq \left \vert \gamma -\beta -\frac{3+\sqrt{5}}{2}\right \vert \left
\vert \gamma -\beta -\frac{3-\sqrt{5}}{2}\right \vert \\
& \leq \left( \left \vert \gamma -\beta -\frac{3-\sqrt{5}}{2}\right \vert +%
\sqrt{5}\right) \left \vert \gamma -\beta -\frac{3-\sqrt{5}}{2}\right \vert
\\
& =\left \vert \gamma -\beta -\frac{3-\sqrt{5}}{2}\right \vert ^{2}+\sqrt{5}%
\left \vert \gamma -\beta -\frac{3-\sqrt{5}}{2}\right \vert .
\end{align*}%
Let $s=\left \vert \gamma -\beta -\frac{3-\sqrt{5}}{2}\right \vert ,$ then%
\begin{equation*}
s^{2}+\sqrt{5}s\geq 1,
\end{equation*}%
which gives the solution $s\geq \frac{3-\sqrt{5}}{2}$ and hence we obtain
one inequality 
\begin{equation*}
\frac{3-\sqrt{5}}{2}\leq \left \vert \gamma -\beta -\frac{3-\sqrt{5}}{2}%
\right \vert \leq |\gamma |+\left \vert \beta +\frac{3-\sqrt{5}}{2}\right
\vert .
\end{equation*}%
Next, we can finish the other inequality by rearranging the second factor of
the inequality (\ref{pre-inequ}), 
\begin{align*}
1& \leq \left \vert \gamma -\beta -\frac{3+\sqrt{5}}{2}\right \vert \left
\vert \gamma -\beta -\frac{3-\sqrt{5}}{2}\right \vert \\
& \leq \left \vert \gamma -\beta -\frac{3+\sqrt{5}}{2}\right \vert \left(
\left \vert \gamma -\beta -\frac{3+\sqrt{5}}{2}\right \vert +\sqrt{5}\right)
\\
& =\left \vert \gamma -\beta -\frac{3+\sqrt{5}}{2}\right \vert ^{2}+\sqrt{5}%
\left \vert \gamma -\beta -\frac{3+\sqrt{5}}{2}\right \vert .
\end{align*}%
Let $y=\left \vert \gamma -\beta -\frac{3+\sqrt{5}}{2}\right \vert ,$ then
the above inequality becomes%
\begin{equation*}
y^{2}+\sqrt{5}y\geq 1,
\end{equation*}%
solving it gives $y\geq \frac{3-\sqrt{5}}{2}$ and hence 
\begin{equation*}
\frac{3-\sqrt{5}}{2}\leq \left \vert \gamma -\beta -\frac{3+\sqrt{5}}{2}%
\right \vert \leq |\gamma |+\left \vert \beta +\frac{3+\sqrt{5}}{2}\right
\vert .
\end{equation*}

Next, we consider $f$ of order $10,$ then 
\begin{equation*}
\beta =-4\sin ^{2}\left( \frac{\pi }{10}\right) =-\frac{3-\sqrt{5}}{2}%
\allowbreak \text{or \ }\beta =-4\sin ^{2}\left( \frac{3\pi }{10}\right)
=\allowbreak -\frac{3+\sqrt{5}}{2}
\end{equation*}%
and hence the term $\left \vert \beta +\frac{3\pm \sqrt{5}}{2}\right \vert $
vanishes. It follows that $\gamma =\frac{\sqrt{5}-3}{2}$\ or $\frac{3-\sqrt{5%
}}{2}.$ We choose the first case of $\gamma $ to compute $\beta (\left[ f,g%
\right] ):$ 
\begin{align*}
\beta (\left[ f,g\right] )& =\mathrm{tr}^{2}\left[ f,g\right] -4=\left(
\gamma +2\right) ^{2}-4 \\
& =\left( \frac{\sqrt{5}-3}{2}+2\right) ^{2}-4 \\
& =\left( 2\cos \left( \frac{\pi }{5}\right) \right) ^{2}-4=-4\sin
^{2}\left( \frac{\pi }{5}\right) .
\end{align*}

Thus, $\left[ f,g\right] $ is elliptic of order $5.$ Let $h=gf^{-1}g^{-1},$
then it is elliptics of order $10,$ and so the product $fh=\left[ f,g\right] 
$ is elliptic of order $5.$

Now choose $g$ of order $2$, this gives a ${\mathbb{Z}}_{2}$-extension $%
\Gamma $\ of the group $\left \langle f,\,gfg^{-1}\right \rangle .$ Since $%
\left \langle f,\,gfg^{-1}\right \rangle =\left \langle
f,\,gf^{-1}g^{-1}\right \rangle ,$ $\left \langle f,\,gfg^{-1}\right \rangle 
$ is the $(10,10,5)$ hyperbolic triangle group and hence $\Gamma $ is a
Kleinian group with principal character 
\begin{equation*}
\left( \frac{\sqrt{5}-3}{2},-\frac{3\pm \sqrt{5}}{2},-4\right) ,
\end{equation*}%
which gives the sharpness,%
\begin{equation*}
\left \vert \gamma \right \vert +\left \vert \beta +\frac{3\pm \sqrt{5}}{2}%
\right \vert =\frac{3-\sqrt{5}}{2}.
\end{equation*}
\end{proof}

\begin{theorem}
Suppose that $\left \langle f,\,g\right \rangle $ is a Kleinian group
generated by two elements $f$ and $g$ with principal character $\left(
\gamma ,\beta ,-4\right) .$ Then%
\begin{equation*}
|\gamma |+\left \vert \beta +2+\sqrt{2}\right \vert \geq 0.117875.
\end{equation*}
\end{theorem}

\begin{proof}
We apply the Chebychev polynomial $%
T_{8}(z)=128z^{8}-256z^{6}+160z^{4}-32z^{2}+1,$ then Lemma \ref{Lem
Chebychev ineq} gives%
\begin{eqnarray*}
\left \vert \gamma -\beta -4\right \vert &\leq &2|T_{8}(\frac{1}{2}(\gamma
-\beta -2))-1| \\
&=&\left \vert (\gamma -\beta -2)^{8}-8(\gamma -\beta -2)^{6}+20(\gamma
-\beta -2)^{4}-16(\gamma -\beta -2)^{2}\right \vert .
\end{eqnarray*}

\ Let $x=\gamma -\beta -2,$ then $\gamma -\beta -4=x-2\neq 0$. It follows
that 
\begin{equation*}
\left \vert x-2\right \vert \leq \left \vert x\right \vert ^{2}\left \vert
x-2\right \vert \left \vert x+2\right \vert \left \vert \left(
x^{2}-2\right) \right \vert ^{2}.
\end{equation*}%
Dividing $\left \vert x-2\right \vert ,$%
\begin{align*}
1& \leq \left \vert x\right \vert ^{2}\left \vert x+2\right \vert \left
\vert \left( x^{2}-2\right) \right \vert ^{2} \\
& =|\gamma -\beta -2|^{2}\,|\gamma -\beta |\, \left \vert (\gamma -\beta
-2)^{2}-2\right \vert ^{2}.
\end{align*}%
Let $y=\left \vert \gamma -\beta -2-\sqrt{2}\right \vert ,$ then the
previous inequality becomes 
\begin{equation*}
1\leq \left( y+\sqrt{2}\right) ^{2}\left( y+2+\sqrt{2}\right) y^{2}\left( y+2%
\sqrt{2}\right) ^{2}.
\end{equation*}%
Solving the inequality gives $0.117875\leq y$ and hence 
\begin{equation*}
0.117875\leq \left \vert \gamma -\beta -2-\sqrt{2}\right \vert \leq |\gamma
|+\left \vert \beta +2+\sqrt{2}\right \vert .
\end{equation*}
\end{proof}

\end{document}